\documentclass{amsart}

\usepackage[utf8]{inputenc}
\usepackage[english]{babel}
\usepackage{amsmath}
\usepackage{amsthm}
\usepackage{amssymb}
\usepackage{array}
\usepackage{diagbox}
\usepackage{dsfont}
\usepackage{etoolbox}
\usepackage{faktor}
\usepackage{hyperref}
\usepackage{cleveref}
\usepackage{mathdots}
\usepackage{mathtools}
\usepackage{multirow}
\usepackage{stackrel}
\usepackage{stmaryrd}
\usepackage{subfigure}
\usepackage{tikz}

\usetikzlibrary{patterns,arrows,cd,calc,matrix,decorations.pathreplacing}

\title[The cohomology rings of the configuration spaces of the torus]{The cohomology rings of the unordered configuration spaces of the torus}
\author[R. Pagaria]{Roberto Pagaria}
\address{Roberto Pagaria}
\address{Scuola Normale Superiore\\ Piazza dei Cavalieri 7, 56126 Pisa\\ Italia}
\email{roberto.pagaria@sns.it}

\newtheorem{thm}{Theorem}[section]

\newtheorem{cor}[thm]{Corollary}
\newtheorem{lemma}[thm]{Lemma}
\theoremstyle{definition}
\newtheorem{definition}[thm]{Definition}
\newtheorem{example}[thm]{Example}
\theoremstyle{remark}
\newtheorem{remark}[thm]{Remark}

\newcommand*\N{\mathbb{N}}
\newcommand*\Z{\mathbb{Z}}
\newcommand*\C{\mathbb{C}}
\newcommand*\Q{\mathbb{Q}}

\newcommand*\V{\mathbb{V}}
\newcommand*\Sn{\mathfrak{S}_n}
\newcommand{\defeq}{\stackrel{\textnormal{def}}{=}}

\DeclareMathOperator{\SL}{SL_2}
\DeclareMathOperator{\slC}{\mathfrak{sl}_2(\C)}
\newcommand*\dga{dga}
\newcommand*\Confn{\mathcal{F}^n(E)}
\DeclareMathOperator{\Fix}{Fix}

\DeclareMathOperator{\Ind}{Ind}
\newcommand\Mn{\mathcal{M}^n(E)}
\DeclareMathOperator{\MCG}{MCG}
\DeclareMathOperator{\Res}{Res}
\DeclareMathOperator{\sgn}{sgn}
\newcommand{\UA}{U\!A}
\newcommand{\UB}{U\!B}
\newcommand{\co}{\colon \!}
\newcommand*\Uconfn{\mathcal{C}^n(E)}

\newcolumntype{C}[1]{>{\centering\arraybackslash}p{#1}}

\newcommand{\ie}{ie}

\makeatletter
\providecommand*{\diff}%
	{\@ifnextchar^{\DIfF}{\DIfF^{}}}
\def\DIfF^#1{%
	\mathop{\mathrm{\mathstrut d}}%
		\nolimits^{#1}\gobblespace}
\def\gobblespace{%
	\futurelet\diffarg\opspace}
\def\opspace{%
	\let\DiffSpace\!%
	\ifx\diffarg(%
		\let\DiffSpace\relax
	\else
		\ifx\diffarg[%
			\let\DiffSpace\relax
		\else
			\ifx\diffarg\{%
				\let\DiffSpace\relax
			\fi\fi\fi\DiffSpace}
\robustify{\diff}

\newcommand{\dd}{\diff}

\makeatletter
\newcommand*{\bigcdot}{%
  {\mathbin{\mathpalette\bigcdot@{}}}%
}
\newcommand*{\bigcdot@scalefactor}{.75}
\newcommand*{\bigcdot@widthfactor}{1.4}
\newcommand*{\bigcdot@}[2]{%
  \sbox0{$#1\vcenter{}$}% math axis
  \sbox2{$#1\cdot\m@th$}%
  \hbox to \bigcdot@widthfactor\wd2{%
    \hfil
    \raise\ht0\hbox{%
      \scalebox{\bigcdot@scalefactor}{%
        \lower\ht0\hbox{$#1\bullet\m@th$}%
      }%
    }%
    \hfil
  }%
}
\makeatother
\robustify{\bigcdot}

\makeatletter 
\newcommand\mynobreakpar{\par\nobreak\@afterheading} 
\makeatother

\begin{document}
\begin{abstract}
We study the cohomology ring of the configuration space of unordered points in the two dimensional torus.
In particular, we compute the mixed Hodge structure on the cohomology, the action of the mapping class group, the structure of the cohomology ring and we prove the formality over the rationals.
\end{abstract}

\maketitle

\section*{Introduction}
We fully describe the cohomology with rational coefficients of the configuration spaces of unordered points in an elliptic curve (frequently called torus).

Configuration spaces of points are related to physics (state spaces of non-colliding particles on a manifold), robotics (motion planning), knot theory, and topology.
Configuration spaces give invariants of the homeomorphism type of the base space.
In the algebraic setting, configuration spaces are open in the moduli spaces of points.

Since the literature is very extensive, we compare our work only with the main results on the (co-)homology of configuration spaces.
The first computation of the cohomology algebra of configuration spaces is due to Arnol'd \cite{Arnold69,Arnold70} in the case of $\mathbb{R}^2$.
This result has been generalized by Cohen, Lada, and May \cite{Cohen76} to the configuration space of $\mathbb{R}^n$ and later by Goresky and Macpherson \cite{GM83}.
Partially additive results have been obtained: by B\" odigheimer and Cohen \cite{BC88} for once-punctured oriented surfaces, by the same authors and Taylor \cite{BCT89} for odd dimensional manifolds, and by Drummond-Cole and Knudsen \cite{DCK2017} for surfaces in general.
However there is no description of the ring structure; we provide it in the case of elliptic curves.
The Betti numbers $\mathcal{C}^n(X)$ are described in the following cases: for $X=\mathbb{P}^2(\mathbb{R})$ by Wang \cite{WangP2R}, for $X$ a sphere by Salvatore \cite{Sal04}, for $X=\mathbb{P}^2(\C)$ by Felix and Tanré \cite{FT05} and for elliptic curves by Maguire and Schiessl \cite{Maguire2016,Schiessl16}.

In this paper we improve the previous results on configuration spaces in an elliptic curve in three ways.
We describe:%\mynobreakpar
\begin{itemize}
\item the mixed Hodge structure on the cohomology (\Cref{thm:MHP_Groth_ring}),
\item the action of the mapping class group (\Cref{thm:MHP_Groth_ring}),
\item the ring structure (\Cref{thm:main_ring}).
\end{itemize}
The formality result over the rationals is proven in \Cref{cor:formality}.

We prove these results using the Kri\v z model \cite{Kriz94, Totaro, Bibby16, Dupont15} and the representation theory on it \cite{AAB14, Azam15}.

In \Cref{sect:model} we recall the Kri\v z model, then in \Cref{section2} we improve the result on the decomposition of the Kri\v z model into irreducible representations, see \Cref{thm:strong_dec_E2}.
Descriptions of the mixed Hodge structure and of the action of the mapping class group are obtained in \Cref{section3} by computing the cohomology of the model.
Finally, the ring structure is presented in the last section.

\section{The Kri\v z model} \label{sect:model}

Let $E$ be an elliptic curve and consider the configuration space of $n$ ordered distinct points
\[\Confn \defeq \{ \underline{p} \in E^n \mid p_i \neq p_j \}.\]

The symmetric group $\Sn$ acts on $\Confn$ by permuting the coordinates and the quotient is the configuration space of $n$ unordered points
\[ \Uconfn \defeq \Confn/\Sn. \]

We also consider the space $\Mn$, defined by
\[ \Mn \defeq \big\{ \underline{p} \in \Confn \mid \sum p_i =0 \big\}.\]
Notice that there exists a non canonical isomorphism $\Confn \cong E \times \Mn$.

In this section we recall a rational model for the cohomology algebra of $\Confn$.
The model is a commutative differential bi-graded algebra (\dga) that can be obtain in two different ways: as a specialization of the Kri\v z model for the configuration spaces or as the second page of the Leray spectral sequence (also known as the Totaro spectral sequence) for elliptic arrangements.
Our main references for the first approach are \cite{Kriz94,AAB14,Azam15} and for the second one are \cite{Totaro,Dupont15,Bibby16}.
In the following we define the models for the cohomology of $\Confn$ and of $\Mn$.

Let $\Lambda$ be the exterior algebra over $\Q$ with generators
\[\{x_i,y_i,\omega_{i,j}\}_{1 \leq i < j \leq n}.\]
We set the degree of each $x_i$ and $y_i$ equal to $(1,0)$ and the degree of $\omega_{i,j}$ equal to $(0,1)$.
Define the differential $\dd \co \Lambda \to \Lambda$ of bi-degree $(2,-1)$ on generators as follows: $\dd x_i=0$ and $\dd y_i=0$  for $i=1, \dots, n$ and 
\[ \dd \omega_{i,j} \defeq (x_i-x_j)(y_i-y_j).\]
For the sake of notation we set $\omega_{i,j}:=\omega_{j,i}$ for $i>j$.
 
We define the \dga\ $A^{\bigcdot,\bigcdot}$ as the quotient of $\Lambda$ by the following relations:
\begin{align*}
%\label{eq:x_i_omega_i}
(x_i-x_j)\omega_{i,j}=0 \quad \textnormal{and} \quad (y_i-y_j)\omega_{i,j}=0, \\
%\label{eq:circuit}
\omega_{i,j}\omega_{j,k}-\omega_{i,j}\omega_{k,i}+\omega_{j,k} \omega_{k,i} =0 .
\end{align*}
Notice that the ideal is preserved by the differential map, thus the differential $\dd \co A^{\bigcdot,\bigcdot} \to A^{\bigcdot,\bigcdot}$ is well defined.

\begin{remark}
The model $A^{\bigcdot,\bigcdot}$ coincides with the Kri\v z model $E_\bigcdot^\bigcdot$ introduced in \cite{Kriz94} up to shifting the degrees, \ie\
\[ A^{p,q} \cong E^{p+q}_q.\]
The \dga\ $E_\bigcdot^\bigcdot$ is a rational model for $X$, as shown in \cite[Theorem 1.1]{Kriz94}.
\end{remark}

In order to study the cohomology of $A^{\bigcdot,\bigcdot}$ we need to introduce the elements $u_{i,j}=x_i-x_j$, $v_{i,j}=y_i-y_j$ and $\gamma=\sum_{i=1}^n x_i$, $\overline{\gamma}=\sum_{i=1}^n y_i \in A^{1,0}$.

We define the \dga\ $B^{\bigcdot,\bigcdot}$ as the subalgebra of $A^{\bigcdot,\bigcdot}$ generated by $u_{i,j}, v_{i,j}$ and $\omega_{i,j}$ for $1\leq i < j \leq n$.
Let $D^{\bigcdot, 0}$ be the subalgebra of $A^{\bigcdot,\bigcdot}$ generated by $\gamma$ and $\overline{\gamma}$ endowed with the zero differential map.
Notice that 
\begin{equation} \label{eq:iso}
A^{\bigcdot,\bigcdot} \cong B^{\bigcdot,\bigcdot} \otimes_\Q D^{\bigcdot,0}
\end{equation}
as differential algebras and  that $D^{\bigcdot,0}$ is the cohomology ring of the elliptic curve $E$.

The mixed Hodge structure on the cohomology of algebraic varieties defines a bigrading compatible with the algebra structure (see~\cite[p.81]{DelUtile} or \cite[Theorem 8.35]{VoisinI}).
In our case the bigrading given by the mixed Hodge structure coincides with the one given by the Leray spectral sequence as shown by Totaro \cite[Theorem 3]{Totaro} and by Gorinov \cite{Gorinov17}.
Explicitly, the subspace $A^{p,q}$ has weight $p+2q$ and degree $p+q$.

The following result is a particular case of \cite[Theorem 3.3]{Bibby16} and of \cite[Theorem 1.2]{Dupont15}.
\begin{thm}
The cohomology algebra of $\Confn$ (or of $\Mn$) with rational coefficients is isomorphic to the cohomology of the \dga\ $A^{\bigcdot,\bigcdot}$ (respectively of $B^{\bigcdot,\bigcdot}$).
Moreover, the $n^2$-sheeted covering 
\begin{align*}
E \times \Mn & \to \Confn \\
(q, \underline{p}) \qquad & \mapsto (p_i+q)_{i=1,\dots,n}
\end{align*}
induces the isomorphism of eq.~\eqref{eq:iso}.
\end{thm}

\section{Representation theory on the Kri\v z model} \label{section2}

Now we study the action of the symmetric group $\Sn$ and of $\SL(\Q)$ on the algebras $A^{\bigcdot,\bigcdot}$ and $B^{\bigcdot,\bigcdot}$.
Those actions are given by a geometric action on $\Confn$.
For general reference about the representation theory of the Lie group and of the Lie algebra we refer to \cite{Hall} and to \cite{FH91}, respectively.
The cases of $\SL(\C)$ and of $\slC$ can be found in \cite{GWbook}.

\subsection{Definition of the actions}
\label{subsec:def_actions}

Consider the action of $\Sn$ on $\Confn$ defined by
\[ \sigma^{-1} \cdot (p_1, \dots, p_n)= (p_{\sigma(1)}, \dots, p_{\sigma(n)})\]
for all $\sigma \in \Sn$.
This induces an action on $A^{\bigcdot,\bigcdot}$ and on $B^{\bigcdot,\bigcdot}$ defined by
\begin{align*}
&\sigma^{-1}(x_i)=x_{\sigma(i)}, \\
&\sigma^{-1}(y_i)=y_{\sigma(i)}, \\
&\sigma^{-1}(\omega_{i,j})=\omega_{\sigma(i), \sigma(j)}
\end{align*}
for all $1\leq i<j\leq n$ and all $\sigma \in \Sn$.

The mapping class group $\MCG(E)$ acts naturally on $\Confn$ and on $\Uconfn$.
\begin{thm}[Theorem 2.5 \cite{BDprimer}]
The mapping class group $\MCG(E)$ of the torus is isomorphic to $\SL(\Z)$ and the isomorphism is given by the natural action of $\MCG(E)$ on $H^1(E;\Z)$.
\end{thm}

Let $f$ be an automorphism of $E$, the map induces the following vertical morphisms
{\setlength\mathsurround{0pt}
\begin{center}
\begin{tikzcd}
\Confn \arrow[r, hook] \arrow[d, "f^n_{|\Confn}" left] & E^n \arrow[d, "f^n"] \\
\Confn \arrow[r, hook] & E^n
\end{tikzcd}
\end{center}
}
and by functoriality of the Leray spectral sequence it induces the action of $\SL(\Z)$ on $A^{\bigcdot,\bigcdot}$.
We explicitly describe this action on the generators $\omega_{i,j}$, $x_i$, and $y_i$: since $f^n \co E^n \to E^n$ fixes the divisor $\{p_i=p_j\}$, then $f \cdot \omega_{i,j} =\omega_{i,j}$.
The other generators belongs to $A^{1,0}= H^1(E^n) \cong H^1(E)^{\otimes n}$.
Therefore the action of $\MCG(E) \cong \SL (\Z)$ on $A^{1,0}$ is given by 
\[ \left( \begin{array}{cc}
a & b \\ 
c & d
\end{array} \right) \cdot x_i = ax_i+cy_i \quad \quad
\left( \begin{array}{cc}
a & b \\ 
c & d
\end{array} \right) \cdot y_i = bx_i+dy_i. \]
This action extends to $\SL(\Q)$ and since the actions of $\Sn$ and of $\SL(\Q)$ commute, then $A^{\bigcdot, \bigcdot}$, $B^{\bigcdot, \bigcdot}$ and $D^{\bigcdot,0}$ become $\Sn \times \SL(\Q)$-modules.

\subsection{Decomposition into \texorpdfstring{$\Sn$}{Sn}-representations}
We recall a result of \cite[Theorem 3.15]{AAB14} on the decomposition of $A^{\bigcdot,\bigcdot}$ into $\Sn$-modules.
The notations used here follow the ones in \cite{AAB14}.

Let $L_*=(\lambda_1, \dots, \lambda_t)$ be a partition of the number $n$, \ie\ $\lambda_i \in \N_+$ and $\sum_{i=1}^t \lambda_i=n$.
We mark all blocks with labels in $\{1,x,y,xy\}$, an ordered basis of $H^\bigcdot(E)$.
The order is $1 \prec x \prec y \prec xy$.

\begin{definition}
A \textit{marked partition} $(L_*,H_*)$ is a partition $L_* \vdash n$ together with marks $H_*=(h_1,\dots, h_t)$, $h_i \in \{1,x,y,xy\}$ such that: if $\lambda_i = \lambda_{i+1}$ then $h_i \succeq h_{i+1}$.
\end{definition}

Let $C_k$ be the cyclic group of order $k$.
For any partition $L_* \vdash n$ define $C_{L_*}$ as the product of the cyclic groups $C_{\lambda_i}$ for $i=1,\dots, t$.
It acts on $\{1, \dots, n\}$ in the natural way.
Consider a marked partition $(L_*,H_*)$ and define $N_{L_*,H_*}$ as the group that permutes the blocks of $L_*$ with the same labels.
The group $N_{L_*,H_*}$ is a product of symmetric groups.
Call $Z_{L_*,H_*}$ the semidirect product $C_{L_*}\rtimes N_{L_*,H_*}$.

\begin{example}\label{example:label_part}
Let $(L_*,H_*)$ be the marked partition $L_*=(5,5,5,5,1,1,1)\vdash 23$ and $H_*=(xy,xy,xy,1,x,x,x)$.
The group $C_{L_*}\cong (\Z/5\Z)^{4} < \mathfrak{S}_{23}$ is generated by $(1,2,3,4,5),(6,7,8,9,10),(11,12,13,14,15)$, and $(16,17,18,19,20)$.
The subgroup $N_{L_*,H_*}\cong \mathfrak{S}_3 \times \mathfrak{S}_3$ is generated by the permutations $(1,6)(2,7)(3,8)(4,9)(5,10)$, $(1,11)(2,12)(3,13)(4,14)(5,15)$, $(21,22)$, and $(21,23)$.
Finally, $Z_{L_*,H_*}$ is a group isomorphic to $(\Z/5\Z \wr \mathfrak{S}_3) \times \Z/5\Z \times \mathfrak{S}_3$.
\end{example}

Given two representations $V,W$ of two groups $G$ and $H$ respectively, define the tensor representation $V\boxtimes W$ of $G\times H$ by the vector space $V\otimes W$ with the action $(g,h)(v\otimes w)= g(v) \otimes h(w)$.

We define the following one-dimensional representations.
Let $\varphi_n$ be a faithful character of the cyclic group and $\varphi_{L_*}$ the character of $C_{L_*}\cong \Z/\lambda_1\Z \times \cdots \times \Z/\lambda_t \Z$ given by
\[ \varphi_{L_*} \defeq \sgn_n|_{C_{L_*}}\cdot (\varphi_{\lambda_1} \boxtimes \dots \boxtimes \varphi_{\lambda_r}).\]

Recall that the degree $\deg$ of $1,x,y,xy$ are respectively $0,1,1,2$.
Let $\alpha_{L_*,H_*}$ be the one dimensional representation of $N_{L_*,H_*} \cong \mathfrak{S}_{\mu_1} \times \cdots \times \mathfrak{S}_{\mu_l}$ defined on generators by 
\[ \alpha_{L_*,H_*}(\sigma) \defeq 
(-1)^{\lambda+\deg(h)+1},\] %\todo{remark su Lie operad}
where $\sigma$ is the permutation that exchange two blocks of size $\lambda$ and label $h$.
Set $\xi_{L_*,H_*}$ to be the one dimensional representation of $Z_{L_*,H_*}$ such that $\Res_{C_{L_*}}^{Z_{L_*,H_*}} \xi_{L_*,H_*}=\varphi_{L_*} $ and $ \Res_{N_{L_*,H_*}} ^{Z_{L_*,H_*}} \alpha_{L_*,H_*}$ .

We define $|L_*|=n-t$ for a partition $L_*=(\lambda_1, \dots, \lambda_t)$ of $n$ and for a mark $H_*$ the numbers $|H_*|=\sum_{i=1}^t \deg (h_i)$ and $\| H_*\| = |\{i \mid h_i=x\}|-|\{i \mid h_i=y\}|$.

\begin{thm}[{\cite[Theorem~3.15]{AAB14}}]\label{thm:weak_dec_E2}
There exist $\Sn$-representations $A_{L_*,H_*}\subset A^{p,q}$ such that 
 \[A^{\bigcdot,\bigcdot}=\bigoplus_{\substack{|L_*|=q \\ |H_*|=p}} A_{L_*,H_*} \] as $\Sn$-representation.
Moreover:
\[
 A_{L_*,H_*} \otimes_\Q \C \simeq  \Ind _{Z_{L_*,H_*}}^{\Sn} \xi_{L_*,H_*}.
\]
\end{thm}

\begin{example}
Consider the marked partition $(L_*,H_*)$ of \Cref{example:label_part}, the characters are shown in the following table.
\[ \begin{tabular}{c|c|c|c|c}
 & $(1,2,3,4,5)$ & $(16,17,18,19,20)$ & $(1,6)(2,7)(3,8)(4,9)(5,10)$ & $(21,22)$\\ 
$\varphi$ & $\zeta_5$ &  $\zeta_5$ & & \\ 
\hline 
$\alpha$ & & & $1$ & $-1$ \\ 
\hline 
$\xi$ & $\zeta_5$ & $\zeta_5$ & $1$ & $-1$ \\ 
\end{tabular}\]
\end{example}

\subsection{Decomposition into \texorpdfstring{$\Sn \times \SL(\Q)$}{Sn x SL2(Q)}-representations} \label{subsect:action_SL2}

Let $T=\{H_t\} \cong \Q^* $ be the maximal torus in $\SL (\Q)$ generated by the diagonal matrices $H_t=\left(
\begin{smallmatrix}
t & 0 \\
0 & t^{-1}
\end{smallmatrix}\right)$.
Let $\V_1$ be the irreducible representation $\Q^2$ with the standard action of matrix-vector multiplication and let $\V_k= S^k \V_1$ be the irreducible representation given by the symmetric power of $\V_1$.
The representation $\V_k$ has dimension $k+1$ and can be view as $\Q[x,y]_k$, \ie \ the vector space of homogeneous polynomials in two variables.
The action of $T$ on the monomials is given by $H_t \cdot x^ay^{k-a} = t^{2a-k} x^ay^{k-a}$, thus $\V_k$ decomposes, as representations of $T$
\begin{equation} \label{eq:spiegazione_per_C}
\V_k= \bigoplus_{a=0}^k V(2a-k),
\end{equation}
where $V(2a-k)$ is the subspace where $H_t$ acts with character $t^{2a-k}$, \ie \ the subspace generated by $x^ay^{k-a}$.
Since the group $\SL (\Q)$ is dense in $\SL (\C)$, each irreducible regular representation of $\SL (\Q)$ is isomorphic to $\V_k$ for some $k\in \N$.
For a proof see \cite[Proposition 2.3.5]{GWbook} and use a density reasoning.

As a consequence we can decompose a representation $V$ of $\SL (\Q)$ using its decomposition $V= \oplus_{a \in \Z} V(a)^{\oplus n_a}$ as representation of $T$: indeed $V\cong \oplus_{k \in \N} \V_k^{\oplus m_k}$ as representation of $\SL (\Q)$, where $m_k=n_k-n_{k+2}$.
By setting $V=\V_m \otimes \V_n$, we obtain the following formula for $n\leq m$:
\[ \V_m \otimes \V_n \cong \V_{m+n} \oplus \V_{m+n-2} \oplus \dots \oplus \V_{m-n}. \]

As observed in \Cref{subsec:def_actions}, the group $\SL(\Q)$ acts trivially on $\omega_{i,j}$ for all $1\leq i<j\leq n$ and the two dimensional subspace generated by $x_i$ and $y_i$ is isomorphic to $\V_1$ as representation of $\SL(\Q)$.

We will use the decomposition of \Cref{thm:weak_dec_E2} to obtain a decomposition of $A^{\bigcdot,\bigcdot}$ into $\Sn \times \SL(\Q)$-modules.
Let 
\[A^{p,q}_a= \bigoplus_{\substack{|L_*|= q \\ |H_*|=p, \;  \|H_*\|=a}} A_{L_*,H_*}\]
be a $\Sn \times T$-stable subspace of $A^{p,q}$ and hence we have $A^{p,q}= \oplus_{a=-p}^{p} A^{p,q}_a$.
Let $p_a\co A^{p,q} \to A^{p,q}_a$ be the $\Sn \times T$-equivariant projection and define $Y=\left(
\begin{smallmatrix}
1 & 0 \\
1 & 1
\end{smallmatrix}\right) \in \SL(\Q)$ and, for $a\geq 0$, define $\pi_a \co A^{p,q}_{a+2} \to A^{p,q}_a$ by $v \mapsto p_a(Y\cdot v)$.

Notice that, $\pi_a$ is a morphism of $\Sn$-representations and that $A^{p,q}_a$ is zero if $a \not \equiv p \mod 2$, or if $a>p$, or if $a > n-q$.

\begin{lemma}
The map $\pi_a$ is injective.
\end{lemma}

\begin{proof}
If $V\subseteq A^{p,q}$ is a $SL(\Q)$-representation, then $p_a(V) \subseteq V$, thus it is enough to prove that $\pi_a \co V(a+2) \to V(a)$ is injective for all irreducible representation $\V_k \subseteq A^{p,q}$.
If $k\not \equiv a \mod 2$ or if $k<a+2$ then $V(a+2)=0$.
Otherwise $k=a+2b+2$ for some $b\geq 0$ and $V(a+2)$ is a one dimensional vector space generated by the homogeneous monomial $x^{a+b+2}y^{b}$.
The projection $p_a$ has kernel equal to $\oplus_{l \neq a} V(l)$, thus
\[ p_a(Y \cdot x^{a+b+2}y^{b})= p_a ((x+y)^{a+b+2}y^{b})=(a+b+2) x^{a+b+1} y^{b+1}. \]
This last term is non-zero since $a,b \geq 0$ and therefore $\pi_a$ is injective.
\end{proof}

\begin{thm}\label{thm:strong_dec_E2}
The algebra $A^{\bigcdot,\bigcdot}$ decomposes as $\Sn \times \SL(\Q)$-representation in the following way:
\begin{equation}
A^{p,q} \cong \bigoplus_{a=0}^p \operatorname{coker} \pi_a \boxtimes \V_a.
\end{equation}
\end{thm}

\begin{proof}
Observe that the maximal torus $T$ of $\SL(\Q)$ acts on $A^{p,q}_a$ with character $t^a$, thus by \Cref{thm:weak_dec_E2} we have
\[ A^{p,q} = \bigoplus_{a=-p}^p A^{p,q}_a\]
as $\Sn \times T$ representations.
By using eq~\eqref{eq:spiegazione_per_C} we obtain that $\bigoplus_{a=0}^p \operatorname{coker} \pi_a \boxtimes \V_a$ is isomorphic to $\bigoplus_{a=-p}^p A^{p,q}_a$ as representation of $\Sn \times T$.
The representation theory of $\SL(\Q)$ ensure that the representations $A^{p,q}$ and $\bigoplus_{a=0}^p \operatorname{coker} \pi_a \boxtimes \V_a$ are isomorphic as $\Sn \times \SL(\Q)$-representations.
\end{proof}

Define the $\Sn$-invariant subalgebra of $A^{\bigcdot,\bigcdot}$ by $\UA^{\bigcdot,\bigcdot}$ and of $B^{\bigcdot,\bigcdot}$ by $\UB^{\bigcdot,\bigcdot}$.
Obviously we have $\UA^{\bigcdot,\bigcdot}=\UB^{\bigcdot,\bigcdot} \otimes_\Q D^\bigcdot$.
We use the previous calculation to compute $\UA^{\bigcdot,\bigcdot}$
\begin{cor}
For $q>p+1$ we have $\UA^{p,q}=0$.
\end{cor}

\begin{proof}
Let $\mathds{1}_n$ be the trivial representation of $\Sn$.
We use \Cref{thm:weak_dec_E2} to show that
\[ \langle \mathds{1}_n, A^{p,q} \rangle_{\Sn}= 0 \]
for $q>p+1$.
Indeed, it is enough to prove that 
\[ \langle \mathds{1}_n, \Ind _{Z_{L_*,H_*}} ^{\Sn} \xi _{L_*,H_*} \rangle_{\Sn}= 0 \]
for all $(L_*,H_*)$ with $|L_*|= q$ and $|H_*|=p$.
By Frobenius reciprocity we have 
\[ \langle \mathds{1}_n, \Ind _{Z_{L_*,H_*}} ^{\Sn} \xi _{L_*,H_*} \rangle_{\Sn}= \langle \Res _{Z_{L_*,H_*}} ^{\Sn} \mathds{1}_n, \xi _{L_*,H_*} \rangle_{Z_{L_*,H_*}}\]
Since the representations in the right hand side are one-dimensional the value of $\langle \mathds{1}_n, \Ind _{Z_{L_*,H_*}} ^{\Sn} \xi _{L_*,H_*} \rangle_{\Sn}$ is non zero if and only if $\xi _{L_*,H_*}=\mathds{1}$.

By definition $\xi _{L_*,H_*}=\mathds{1}$ is equivalent to $\varphi _{L_*}=\mathds{1}$ and $\alpha _{L_*,H_*}=\mathds{1}$.
From the fact that $\varphi_k=\sgn_k$ only for $k=1,2$, $\psi _{L_*}=\mathds{1}$ %is equivalent to 
if and only if
$\lambda_i=1,2$ for all $i=1,\dots,t$.
The condition $\alpha _{L_*,H_*}=\mathds{1}$ implies that the only marked blocks of $(L_*,H_*)$ that appear more than once are the ones with $\lambda_i=2$ and $\deg(h_i)=1$ or the ones with $\lambda_i=1$ and $\deg(h_i) \neq 1$.

Consequently, $\langle \mathds{1}_n, \Ind _{Z_{L_*,H_*}} ^{\Sn} \xi _{L_*,H_*} \rangle_{\Sn} \neq 0$ only if $L_*=(2^q,1^{n-2q})$ and the degree of $h_i$ is $1$ for $i<q$, this implies $p\geq q-1$ contrary to our hypothesis.
\end{proof}

\begin{cor}\label{cor:q<p+1}
For $q>p+1$ we have $\UB^{p,q}=0$.\hfill \qedsymbol
\end{cor}

\section{The additive structure of the cohomology} \label{section3}

We compute the cohomology with rational coefficients of the unordered configuration spaces of $n$ points, taking care of the mixed Hodge structure and of the action of $\SL(\Q)$.
The integral cohomology groups are known only for small $n$ in \cite[Table 2]{Napolitano}, where a cellular decomposition of ordered configuration spaces is given.
In this section, we use the calculation of the Betti numbers of $\Uconfn$ to determine the Hodge polynomial in the Grothendieck ring of $\SL(\Q)$.

Observe that $H^\bigcdot(\Uconfn)= H^\bigcdot (\Confn) ^{\Sn}$ by the Transfer Theorem.
Define the series
\[T(u,v)=\frac{1+u^3v^4}{(1-u^2v^3)^2}= 1+2u^2v^3+ u^3v^4+ 3u^4v^6+ 2u^5v^7+ \dots\]
and let $T_n(u,v)$ be its truncation at degree $n$ in the variable $u$.

The computation of the Betti numbers of unordered configuration space of $n$ points in an elliptic curve was done simultaneously by \cite{DCK2017}, \cite{Maguire2016}, and \cite{Schiessl16} in different generality.
We point to the last reference because \cite[Theorem]{Schiessl16} fits exactly our generality.

\begin{thm}\label{thm:Poin_Cn}
The Poincaré polynomial of $\Uconfn$ is $(1+t)^2T_{n-1}(t,1)$.
\end{thm}

We use the notation $V u^k v^h$ to denote a vector space $V$ in degree $k$ with a Hodge structure of weight $h$.
The Grothendieck ring of $\SL(\Q)$ is the free $\Z$-module with basis given by $[V]$ for all finite-dimensional irreducible representations $V$ of $\SL(\Q)$ and product defined by the tensor product of representations.
\begin{definition}
The Hodge polynomial of $\Uconfn$ with coefficients in the Grothendieck ring of $\SL(\Q)$ is 
\[ \sum_{i=0}^{2n} \sum_{k=i}^{2i} \Bigg[ \faktor{W_k H^i(\Uconfn;\Q)}{W_{k-1} H^i(\Uconfn;\Q)} \Bigg] u^i v^k, \]
where $W_k H^i(\Uconfn;\Q)$ is the weight filtration on $H^i(\Uconfn;\Q)$.
The ordinary Hodge polynomial is
\[ \sum_{i=0}^{2n} \sum_{k=i}^{2i} \dim_\Q \Bigg( \faktor{W_k H^i(\Uconfn;\Q)}{W_{k-1} H^i(\Uconfn;\Q)} \Bigg) u^i v^k. \]
\end{definition}

We prove a stronger version of \Cref{thm:Poin_Cn}.
\begin{thm}\label{thm:MHP_Groth_ring}
The Hodge polynomial of $\Uconfn$ with coefficients in the Grothen\-dieck ring of $\SL(\Q)$ is
\begin{equation}
\label{eq:MH poly in G ring}
([\V_0]+ [\V_1] uv+ [\V_0] u^2v^2)\left( \sum_{i=0}^{\lfloor \frac{n-1}{2} \rfloor }  [\V_i] u^{2i}v^{3i}  + \sum_{i=1}^{\lfloor \frac{n}{2} \rfloor -1 }  [\V_{i-1}] u^{2i+1}v^{3i+1} \right)
\end{equation}
and the ordinary Hodge polynomial is $(1+uv)^2 T_{n-1}(u,v)$.
\end{thm}

\Cref{fig:E3_inv} represents the module $H^{\bigcdot,\bigcdot}(\UB)$ that corresponds to the right factor of eq.~\eqref{eq:MH poly in G ring}.

\begin{figure}
\centering
\begin{tabular}{C{.48\textwidth}C{.48\textwidth}}
\subfigure[Case $n=2q+1$ odd.]{
\begin{tikzpicture}[ampersand replacement=\&]
\matrix (m) [matrix of math nodes,
             nodes in empty cells,
             nodes={inner sep=0pt, outer sep=0pt,
                    anchor=base},
             column sep=1ex, row sep=1ex,
             anchor=center
             ]%
{
q \&[2ex] \& \& \& \& \& \V_q \\
q-1 \&[2ex] \& \& \& \& \V_{q-1} \& \V_{q-2} \\
\vdots \&[2ex] \& \& \& \iddots \& \iddots \& \\
[3ex,between origins]
2 \&[2ex] \& \& \V_2 \& \V_1 \& \& \\
1 \&[2ex] \& \V_1 \& \V_0 \& \& \& \\
0 \&[2ex] \V_0 \& \& \& \& \& \\
    [3ex,between origins]
\&[2ex]  0 \& 1 \& 2 \& 3 \& \cdots \& q  \strut \\
};
\draw[thick] ($(m-1-2.north)!0.5!(m-1-1.north)$) -- ($(m-7-2)!0.5!(m-7-1)$) ;
\draw[thick] ($(m-6-1.west)!0.5!(m-7-1.west)$) -- ($(m-6-7.east)!0.5!(m-7-7.east)$);
\end{tikzpicture}
} &
\subfigure [Case $n=2q+2$ even.] {
\begin{tikzpicture}[ampersand replacement=\&]
\matrix (m) [matrix of math nodes,
             nodes in empty cells,
             nodes={inner sep=0pt, outer sep=0pt,
                    anchor=base},
             column sep=1ex, row sep=1ex,
             anchor=center
             ]%
{
q \&[2ex] \& \& \& \& \& \V_q \&\V_{q-1} \\
\vdots \&[2ex] \& \& \& \& \iddots \& \iddots \\
[3ex,between origins] 
3 \&[2ex] \& \& \& \V_3 \& \V_2 \& \\
2 \&[2ex] \& \& \V_2 \& \V_1 \& \& \\
1 \&[2ex] \& \V_1 \& \V_0 \& \& \& \& \\
0 \&[2ex] \V_0 \& \& \& \& \& \& \\
    [3ex,between origins]
\&[2ex]  0 \& 1 \& 2 \& 3 \& \cdots \& q \& q+1  \strut \\
};
\draw[thick] ($(m-1-2.north)!0.5!(m-1-1.north)$) -- ($(m-7-2)!0.5!(m-7-1)$) ;
\draw[thick] ($(m-6-1.west)!0.5!(m-7-1.west)$) -- ($(m-6-8.east)!0.5!(m-7-8.east)$);
\end{tikzpicture}
}
\end{tabular}
\caption{The algebra $H^{\bigcdot,\bigcdot}(\UB )$ as representation of $\SL(\Q)$.}\label{fig:E3_inv}
\end{figure}

\subsection{Some elements in cohomology}
\begin{definition}
Let $\alpha, \overline{\alpha} \in A^{1,1}, \beta \in A^{1,2}$ be the elements
\begin{align*}
\alpha & \defeq \sum_{i,k<h} (x_i-x_k)\omega_{k,h} \\
\overline{\alpha} & \defeq \sum_{i,k<h} (y_i-y_k)\omega_{k,h} \\
\beta & \defeq \sum_{i,j,k<h} (3x_i-x_j-2x_k)(y_j-y_k) \omega_{k,h} 
\end{align*}
where the sum is taken over pairwise distinct indices $i,j,k,h$ with $k<h$.
\end{definition}
Notice that the elements $\alpha$ and $\overline{\alpha}$ are defined only for $n>2$ and $\beta$ for $n>3$.
Remember that $\gamma, \overline{\gamma} \in D^1$ were already defined as $\sum_i x_i$ and  $\sum_i y_i$.

\begin{lemma}\label{lemma:alpha}
The element $\alpha$ belongs to $\UB^{1,1}$, is non-zero, and $\dd \alpha=0$.
\end{lemma}
\begin{proof}
First observe that $\alpha=\sum_{i,k<h} u_{i,k}\omega_{k,h} \in B^{1,1}$.
For all $\sigma \in \Sn $ we have 
\[\sigma \alpha = \sum_{i,k<h} u_{\sigma(i),\sigma(k)}\omega_{\sigma(k),\sigma(h)} =\alpha,\]
since $u_{\sigma(i),\sigma(k)}\omega_{\sigma(k),\sigma(h)}=u_{\sigma(i),\sigma(h)}\omega_{\sigma(k),\sigma(h)}$ in $A$. 
The elements $x_i\omega_{k,h}$ and $x_k\omega_{k,h}$ are linearly independent, so it is enough to observe that the coefficient of $x_3\omega_{1,2}$ is $1$.
This proves that $\alpha \neq 0$.
Finally, we compute $\dd \alpha$:
\begin{align*}
\dd \alpha &= \sum_{i,k<h} x_i \dd \omega_{k,h} - x_k \dd \omega_{k,h} \\
&= \sum_{i,k<h} x_i(x_k-x_h)(y_k-y_h) + x_k x_h(y_k-y_h) \\
&= \sum_{i,k,h} x_i x_k y_k- x_ix_h y_k + x_k x_h y_k \\
&= - \sum_{i,k,h} x_ix_h y_k =0,
\end{align*}
where all sums are taken over pairwise distinct indices and the first two with the additional condition $k<h$.
\end{proof}

\begin{lemma}\label{lemma:beta}
The element $\beta$ belongs to $\UB^{1,2}$, is non-zero, and $\dd \beta=0$.
\end{lemma}

\begin{proof}
Observe that
\[ \beta = \sum_{i,j,k<h} (u_{i,j}+2u_{i,k})v_{j,k} \omega_{k,h} \in B^{1,2}\]
and that $\sigma \beta = \beta$ for all $\sigma \in \Sn $ by the relations $u_{j,k} \omega_{k,h} = u_{j,h} \omega_{k,h}$ and $v_{j,k} \omega_{k,h} = v_{j,h} \omega_{k,h}$.
Consider the map $\varphi \co A \to \Q$ defined on generators by $\varphi (\omega_{1,2})=1$, $\varphi (x_3)=1$ and $\varphi(y_4)=1$ and zero on the other generators.
The map $\varphi$ is well defined and $\varphi(\beta)=3$, thus $\beta \neq 0$.
Using the computation in the proof of \Cref{lemma:alpha}, we can observe that $\dd \big( \sum_{i,j,h<k} 3x_i (y_j-y_k) \omega _{k,h} \big)=0$.
The claim $\dd \beta=0$ follows from:
\begin{align*}
\dd(\beta)&= \dd \Big( \sum_{j,k,h} (x_j+2x_k)(y_j-y_k) \omega_{k,h} \Big)\\
&= \sum_{j, k<h} x_jy_j \dd \omega_{k,h} + x_j y_k (x_k-x_h)y_h -2 x_k y_j x_h (y_k-y_h) -2 x_ky_kx_hy_h\\
&= \sum_{j,k,h} x_jy_j x_k y_k- x_iy_j x_h y_k + x_j y_k x_k y_h -2 x_k y_j x_h y_k - x_ky_kx_hy_h\\
&=0,
\end{align*}
where the indexes of the sums are pairwise distinct.
\end{proof}

\begin{lemma}\label{lemma:a^k}
For $n>2q$ the element $\alpha^q$ is non-zero.
\end{lemma}

\begin{proof}
Let us rewrite $\alpha$ as
$\alpha= \sum_{i,k<h} x_i \omega_{k,h} + (2-n) \sum_{k<h} x_k \omega_{k,h}$.
We show that the coefficient of the monomial $m=x_1\omega_{1,2} x_3\omega_{3,4} \dots x_{2q-1}\omega_{2q-1,2q}$ (defined for $n\geq 2q$) in $\alpha^q$ is non-zero for $n>2q$.
%The monomial $m$ can be rewritten has 
%\[ m= \sgn(\sigma) \prod_{i=1}^q x_{2\tau(i)-r(i)} \omega_{2\sigma(i)-1,2\sigma(i)} \]
%for each $\sigma, \tau \in \mathfrak{S}_q$ and each function $r\co [q]\to \{0,1\}$.

This coefficient is
\[ a_q = q! \sum_{\sigma \in \mathfrak{S}_q} \sgn(\sigma) (2-n)^{|\Fix \sigma|}2^{q-|\Fix \sigma|}.\]
Where $q!$ are the ways to choose each $\omega_{2k-1,2k}$ one from each factors, and since $x_{2i-1}\omega_{2k-1,2k}$ has even degree we can suppose (up to the factor $q!$) that $\omega_{2k-1,2k}$ is taken from the $k$-th factor.
Now $x_{2i-1}$ arises from either $x_{2i-1}$ or $x_{2i}$ of the $\sigma(i)$-th factor for some permutation $\sigma \in \mathfrak{S}_q$.
Since $m= \sgn(\sigma) \prod_{i=1}^q x_{2\sigma^{-1}(k)-1} \omega_{2k-1,2k}$ the contribution has the sign of the permutation $\sigma$.
Finally, in $\alpha$ the monomial $x_{2i-1}\omega_{2k-1,2k}$ has coefficient $(2-n)$ if $i=k$ and $1$ otherwise and the monomial $x_{2i}\omega_{2k-1,2k}$ has coefficient $0$ if $i=k$ and $1$ otherwise.

We claim that
\begin{equation} \label{eq:det}
\sum_{\sigma \in \mathfrak{S}_q} \sgn(\sigma) x^{|\Fix \sigma|}= (x-1)^{q-1}(x+q-1),
\end{equation}
since both sides are the determinant of the matrix
\[\left( \begin{array}{cccc}
x & 1 & \cdots & 1 \\ 
1 & x & \cdots & 1 \\ 
\vdots & \vdots & \ddots & \vdots \\ 
1 & 1 & \cdots & x
\end{array} \right). \]
The left hand side of eq~\eqref{eq:det} is obtained by using the Laplace formula for the determinant and the right hand side by relating the that determinant to the characteristic polynomial $(-t)^{q-1}(q-t)$ of the matrix $\left( \begin{smallmatrix} 1 & \cdots & 1 \\ 
\vdots & \ddots & \vdots \\ 
1 & \cdots & 1 \end{smallmatrix} \right)$.
We use eq~\eqref{eq:det} with $x=\frac{2-n}{2}$ to obtain:
\[ a_q= q!2^q \sum_{\sigma \in \mathfrak{S}_q} \sgn(\sigma) \Big( \frac{2-n}{2} \Big) ^{|\Fix \sigma|}= q!2^q \Big( \frac{-n}{2} \Big) ^{q-1} \Big( \frac{2q-n}{2} \Big) \]
Thus $a_q= (-1)^q q!n^{q-1}(n-2q)$ that is non-zero for $n>2q$.
\end{proof}

\begin{lemma}\label{lemma:a^kb}
For $n>2q+1$ the element $\alpha^{q-1}\beta$ is non-zero.
\end{lemma}

\begin{proof}
Let us rewrite $\beta$ as
\begin{multline*}
\beta= \sum_{i,j,k<h} x_i y_j \omega_{k,h} - 2(n-3) \sum_{i,k<h} (x_iy_k + x_ky_i) \omega_{k,h} - (n-3) \sum_{i,k<h} x_iy_i \omega_{k,h}+ \\
 +2(n-2)(n-3) \sum_{k<h} x_ky_k \omega_{k,h} .
\end{multline*}
Let $b_q$ be the coefficient in $\alpha^{q-1}\beta$ of the monomial 
\[x_1\omega_{1,2} x_3\omega_{3,4} \dots x_{2q-1} \omega_{2q-1,2q} y_{2q+1}.\]
This monomial is defined for $n\geq 2q+1$ and we will show that $b_q \neq 0$ for $n>2q+1$.
The number $b_q$ coincides with the coefficient of the same monomial in the product 
\[ \alpha^{q-1} \Big( \sum_{i,j,k<h} x_i y_j \omega_{k,h} - 2(n-3) \sum_{i,k<h}  x_ky_i \omega_{k,h} \Big).\]
With further manipulation, we obtain that $b_q$ is the coefficient of the above monomial in the expression
\[ 3\alpha^q y_{2q+1} + n \alpha^{q-1} \sum_{k<h} x_k \omega_{k,h} y_{2q+1} .\]
Using the computation in the proof of \Cref{lemma:a^k} we obtain 
\begin{align*}
 b_q &= 3(-1)^q q! n^{q-1}(n-2q)+ nq(-1)^{q-1} (q-1)!n^{q-2}(n-2q+2) \\
 &= 2 (-1)^q q! n^{q-1}(n-2q-1).
\end{align*}
The number $b_q$ is non zero for $n>2q+1$.
\end{proof}

\begin{proof}[Proof of \Cref{thm:MHP_Groth_ring}]
It is enough to prove that the Hodge polynomial of $\UB$ in the Grothendieck ring of $\SL(\Q)$ is 
\[\sum_{i=0}^{\lfloor \frac{n-1}{2} \rfloor } [\V_i] u^{2i}v^{3i}  + \sum_{i=1}^{\lfloor \frac{n}{2} \rfloor -1 }  [\V_{i-1}] u^{2i+1}v^{3i+1}\]
Observe that $\operatorname{Im} \dd^{q,p}=0$ for $q>p+1$ by \Cref{cor:q<p+1}.
From \Cref{lemma:alpha} and \Cref{lemma:a^k} we have that the elements $\alpha^k$ for $2k<n$ generate  as $\SL(\Q)$-module a subspace of dimension at least $k+1$ in $H^{k,k}(\UB,\dd)$.
Analogously, from \Cref{lemma:beta} and \Cref{lemma:a^kb} the elements $\alpha^{k-1}\beta$ for $2k+1<n$ generate as $\SL(\Q)$-module a subspace of dimension at least $k$ in $H^{k,k+1}(\UB,\dd)$.
Since the Betti numbers of $\UB^{\bigcdot,\bigcdot}$ (\Cref{thm:Poin_Cn}) coincides with the above dimensions, we have that $H^{2k}(\UB) \cong \V_k u^{2k}v^{3k}$ and $H^{2k+1}(\UB) \cong \V_{k-1} u^{2k+1}v^{3k+1}$.
\end{proof}

\section{The cohomology ring}
In this section we determine the cup product structure in the cohomology of $\Uconfn$ and we prove the formality result.

In the following we consider graded algebras with an action of $\SL(\Q)$.
We will write $(x_i \mid i \in I)_{\SL(\Q)}$ for the ideal generated by the elements $Mx_i$ for all $M \in \SL(\Q)$ and all $i \in I$.

\begin{thm}\label{thm:main_ring}
The cohomology ring of $\Uconfn$ is isomorphic to
\[ \Lambda^\bigcdot \V_1 \otimes \faktor{S^\bigcdot \V_1[b]}{(a^{\lfloor \frac{n+1}{2} \rfloor}, a^{\lfloor \frac{n}{2} \rfloor}b, b^2)_{\SL(\Q)}},\]
where $a$ is a non-zero degree-one element in $V(1) \subset \V_1$ and $b$ an $\SL(\Q)$-invariant indeterminate of degree $3$.
\end{thm}

\begin{proof}
It is enough to prove that $H^\bigcdot(\UB) \cong S^\bigcdot \V_1[b] / (a^{\lfloor \frac{n+1}{2} \rfloor}, a^{\lfloor \frac{n}{2}\rfloor}b, b^2)_{\SL(\Q)}$.
Define the morphism $\varphi \co S^\bigcdot \V_1[b] / (a^{\lfloor \frac{n+1}{2} \rfloor}, a^{\lfloor \frac{n}{2}\rfloor}b, b^2)_{\SL(\Q)} \to H^\bigcdot(\UB)$  that sends $a,b$ to $\alpha, \beta$ respectively.
It is well defined because $H^{k}(\UB)=0$ for $k\geq n$ and $\beta^2=0$ since it has odd degree.
The map $\varphi$ is surjective since $H^\bigcdot(\UB)$ is generated by $\alpha^i$ and $\alpha^i \beta$ as $\SL(\Q)$-module by \Cref{thm:MHP_Groth_ring}.
A dimensional reasoning shows the injectivity of the map $\varphi$.
\end{proof}

\begin{cor}
The cohomology $H^\bigcdot(\Uconfn)$ is generated as an algebra in degrees one, two and three.
\end{cor}
\begin{proof}
A minimal set of generators is given by $\alpha, \overline{\alpha}, \beta, \gamma, \overline{\gamma}$.
\end{proof}

\begin{cor}\label{cor:formality}
The space $\Uconfn$ is formal over the rationals.
\end{cor}

\begin{proof}
We prove that $\UB$ is formal. 
Consider the subalgebra $K^{\bigcdot, \bigcdot}$ of $\UB^{\bigcdot, \bigcdot}$ generated by $\alpha, \overline{\alpha}, \beta$ endowed with the zero differential.
It is concentrated in degrees $(i,i)$ and $(i,i+1)$ because $\beta^2=0$.
Since $K \cap \operatorname{Im} \dd=0$ (\Cref{cor:q<p+1}), $K \hookrightarrow UB$ is a quasi-isomorphism.
The fact that $K\cong H(UB)$ implies that the algebra $\UB$ is formal.
As a consequence $\UA$ is formal.
The space $\Uconfn$ is formal since our model $\UA$ is equivalent to the Sullivan model.
\end{proof}

\subsection*{Acknowledgements}
I would like to thank the referee for several useful suggestions.

\bibliography{Config_elliptic_curve}{}

\providecommand{\bysame}{\leavevmode\hbox to3em{\hrulefill}\thinspace}
\providecommand{\MR}{\relax\ifhmode\unskip\space\fi MR }
\providecommand{\MRhref}[2]{%
  \href{http://www.ams.org/mathscinet-getitem?mr=#1}{#2}
}
\providecommand{\href}[2]{#2}
\begin{thebibliography}{{Mag}16}

\bibitem[AAB14]{AAB14}
Samia Ashraf, Haniya Azam, and Barbu Berceanu, \emph{Representation theory for
  the {K}ri\v{z} model}, Algebr. Geom. Topol. \textbf{14} (2014), no.~1,
  57--90. \MR{3158753}

\bibitem[Arn69]{Arnold69}
Vladimir~I. Arnold, \emph{The cohomology ring of the colored braid group},
  Mathematical notes of the Academy of Sciences of the USSR \textbf{5} (1969),
  no.~2, 138--140.

\bibitem[Arn14]{Arnold70}
\bysame, \emph{On some topological invariants of algebraic functions},
  pp.~199--221, Springer Berlin Heidelberg, Berlin, Heidelberg, 2014.

\bibitem[Aza15]{Azam15}
Haniya Azam, \emph{Cohomology groups of configuration spaces of {R}iemann
  surfaces}, Bull. Math. Soc. Sci. Math. Roumanie (N.S.) \textbf{58(106)}
  (2015), no.~1, 33--47. \MR{3331058}

\bibitem[BC88]{BC88}
Carl-Friedrich B\"{o}digheimer and Frederick~R. Cohen, \emph{Rational
  cohomology of configuration spaces of surfaces}, Algebraic topology and
  transformation groups ({G}\"{o}ttingen, 1987), Lecture Notes in Math., vol.
  1361, Springer, Berlin, 1988, pp.~7--13. \MR{979504}

\bibitem[BCT89]{BCT89}
Carl-Friedrich B\"{o}digheimer, Fred Cohen, and Larry Taylor, \emph{On the
  homology of configuration spaces}, Topology \textbf{28} (1989), no.~1,
  111--123. \MR{991102}

\bibitem[Bib16]{Bibby16}
Christin Bibby, \emph{Cohomology of abelian arrangements}, Proc. Amer. Math.
  Soc. \textbf{144} (2016), no.~7, 3093--3104. \MR{3487239}

\bibitem[CLM76]{Cohen76}
Frederick~R. Cohen, Thomas~J. Lada, and J.~Peter May, \emph{The homology of
  iterated loop spaces}, Lecture Notes in Mathematics, Vol. 533,
  Springer-Verlag, Berlin-New York, 1976. \MR{0436146}

\bibitem[DCK17]{DCK2017}
Gabriel~C. Drummond-Cole and Ben Knudsen, \emph{Betti numbers of configuration
  spaces of surfaces}, J. Lond. Math. Soc. (2) \textbf{96} (2017), no.~2,
  367--393. \MR{3708955}

\bibitem[Del75]{DelUtile}
Pierre Deligne, \emph{Poids dans la cohomologie des vari\'{e}t\'{e}s
  alg\'{e}briques}, Proceedings of the {I}nternational {C}ongress of
  {M}athematicians ({V}ancouver, {B}. {C}., 1974), {V}ol. 1, Canad. Math.
  Congress, Montreal, Que., 1975, pp.~79--85. \MR{0432648}

\bibitem[Dup15]{Dupont15}
Cl\'{e}ment Dupont, \emph{The {O}rlik-{S}olomon model for hypersurface
  arrangements}, Ann. Inst. Fourier (Grenoble) \textbf{65} (2015), no.~6,
  2507--2545. \MR{3449588}

\bibitem[FH91]{FH91}
William Fulton and Joe Harris, \emph{Representation theory}, Graduate Texts in
  Mathematics, vol. 129, Springer-Verlag, New York, 1991, A first course,
  Readings in Mathematics. \MR{1153249}

\bibitem[FM12]{BDprimer}
Benson Farb and Dan Margalit, \emph{A primer on mapping class groups},
  Princeton Mathematical Series, vol.~49, Princeton University Press,
  Princeton, NJ, 2012. \MR{2850125}

\bibitem[FT05]{FT05}
Yves F\'{e}lix and Daniel Tanr\'{e}, \emph{The cohomology algebra of unordered
  configuration spaces}, J. London Math. Soc. (2) \textbf{72} (2005), no.~2,
  525--544. \MR{2156668}

\bibitem[GM88]{GM83}
Mark Goresky and Robert MacPherson, \emph{Stratified {M}orse theory},
  Ergebnisse der Mathematik und ihrer Grenzgebiete (3) [Results in Mathematics
  and Related Areas (3)], vol.~14, Springer-Verlag, Berlin, 1988. \MR{932724}

\bibitem[{Gor}17]{Gorinov17}
A.~G. {Gorinov}, \emph{{A purity theorem for configuration spaces of smooth
  compact algebraic varieties}}, arXiv e-prints (2017), 1--10.

\bibitem[GW09]{GWbook}
Roe Goodman and Nolan~R. Wallach, \emph{Symmetry, representations, and
  invariants}, Graduate Texts in Mathematics, vol. 255, Springer, Dordrecht,
  2009. \MR{2522486}

\bibitem[Hal15]{Hall}
Brian Hall, \emph{Lie groups, {L}ie algebras, and representations}, second ed.,
  Graduate Texts in Mathematics, vol. 222, Springer, Cham, 2015, An elementary
  introduction. \MR{3331229}

\bibitem[Kri94]{Kriz94}
Igor Kri\v{z}, \emph{On the rational homotopy type of configuration spaces},
  Ann. of Math. (2) \textbf{139} (1994), no.~2, 227--237. \MR{1274092}

\bibitem[{Mag}16]{Maguire2016}
Megan {Maguire}, \emph{{Computing cohomology of configuration spaces}}, arXiv
  e-prints (2016), 55.

\bibitem[Nap03]{Napolitano}
Fabien Napolitano, \emph{On the cohomology of configuration spaces on
  surfaces}, J. London Math. Soc. (2) \textbf{68} (2003), no.~2, 477--492.
  \MR{1994695}

\bibitem[Sal04]{Sal04}
Paolo Salvatore, \emph{Configuration spaces on the sphere and higher loop
  spaces}, Math. Z. \textbf{248} (2004), no.~3, 527--540. \MR{2097373}

\bibitem[{Sch}16]{Schiessl16}
Christoph {Schiessl}, \emph{{Betti numbers of unordered configuration spaces of
  the torus}}, arXiv e-prints (2016), 7.

\bibitem[Tot96]{Totaro}
Burt Totaro, \emph{Configuration spaces of algebraic varieties}, Topology
  \textbf{35} (1996), no.~4, 1057--1067. \MR{1404924}

\bibitem[Voi07]{VoisinI}
Claire Voisin, \emph{Hodge theory and complex algebraic geometry. {I}}, english
  ed., Cambridge Studies in Advanced Mathematics, vol.~76, Cambridge University
  Press, Cambridge, 2007, Translated from the French by Leila Schneps.
  \MR{2451566}

\bibitem[Wan02]{WangP2R}
Jeffrey~H. Wang, \emph{On the braid groups for {${\bf R}{\rm P}^2$}}, J. Pure
  Appl. Algebra \textbf{166} (2002), no.~1-2, 203--227. \MR{1868546}

\end{thebibliography}
\bibliographystyle{amsalpha}
\end{document}